\documentclass[a4paper, 12pt]{amsart}
\usepackage{tikz,color}
\usetikzlibrary{decorations.markings}
\usetikzlibrary{shapes, decorations.text}

\newtheorem{thm}{Theorem}

\newtheorem{lem}[thm]{Lemma}

\theoremstyle{definition}
\newtheorem{dfn}[thm]{Definition}
\newtheorem{ntn}[thm]{Notation}

\theoremstyle{remark}
\newtheorem{exm}[thm]{Example}
\newtheorem{rmk}[thm]{Remark}

\numberwithin{equation}{section}

\newcommand{\Aut}{\operatorname{Aut}}

\newcommand{\lsp}{\operatorname{span}} \newcommand{\clsp}{\overline{\lsp}}
\newcommand{\Ext}{\operatorname{Ext}}

\newcommand{\CC}{\mathbb{C}}
\newcommand{\NN}{\mathbb{N}}

\newcommand{\RR}{\mathbb{R}}
\newcommand{\TT}{\mathbb{T}}
\newcommand{\ZZ}{\mathbb{Z}}

\newcommand{\Bb}{\mathcal{B}}

\newcommand{\Kk}{\mathcal{K}}

\newcommand{\Tt}{\mathcal{T}}

\def\inv{^{-1}}

\title[Reconstructing graphs from Toeplitz algebras]{Reconstructing directed graphs from generalised gauge actions on their Toeplitz algebras}

\author{Nathan Brownlowe}
\address[Nathan Brownlowe]{School of Mathematics and Statistics, The University of Sydney, NSW 2006, Australia}
\email{nathan.brownlowe@sydney.edu.au}
\author{Marcelo Laca}
\address[Marcelo Laca]{Department of Mathematics and Statistics, University of Victoria, Victoria, BC V8W 3P4, Canada}
\email{laca@math.uvic.ca}
\author{Dave Robertson}
\address[Dave Robertson]{School of Mathematical and Physical Sciences, University of Newcastle, University Drive,
Callaghan 2308, Australia}
\email{dave84robertson@gmail.com}
\author{Aidan Sims}
\address[Aidan Sims]{School of Mathematics and Applied Statistics, University of Wollongong, NSW 2522, Australia}
\email{asims@uow.edu.au}

\date{\today}

\thanks{This research was supported by Australian Research Council grant DP180100595.
We are grateful to S{\o}ren Eilers for very helpful conversations that led to Example~\ref{ex:not iso}.}

\begin{document}

\begin{abstract}
We show how to reconstruct a finite directed graph $E$ from its Toeplitz algebra, its gauge
action, and the canonical finite-dimensional abelian subalgebra generated by the vertex
projections. We also show that if $E$ has no sinks, then we can recover $E$ from its Toeplitz
algebra and the generalised gauge action that has, for each vertex, an independent copy of the
circle acting on the generators corresponding to edges emanating from that vertex. We show by
example that it is not possible to recover $E$ from its Toeplitz algebra and gauge action alone.
\end{abstract}

\maketitle

\section{Introduction}

In recent years, there has been an enormous amount of work, led by Eilers and his collaborators
({see,} for example, \cite{EilersEtAl:OAD2013, EilersEtAl:xx16, EilersEtAl:MA2017, EilersEtAl:CJM2018,
EilersEtAl:MJM2012, EilersTomforde:MA10, Sorensen:ETDS2013}) on determining which moves on finite
directed graphs generate the equivalence relations determined by various types of isomorphism of
the associated $C^*$-algebras. One spectacular example of this is
\cite[Theorem~3.1]{EilersEtAl:xx16}: if $E$ and $F$ are graphs with finitely many vertices, then
the graph $C^*$-algebras $C^*(E)$ and $C^*(F)$ are stably isomorphic if and only if $E$ can be
transformed into $F$ using a finite sequence of in-splittings, out-splittings, reductions,
additions of sinks, Cuntz splices, Pulelehua moves, and the inverses of these moves.

By contrast, relatively little attention has been paid to the Toeplitz algebras of directed
graphs, until the recent interest in KMS-theory (see, for example, \cite{CarlsenLarsen:JFA2016,
aHLRS3, KajiwaraWatatani:KJM2013, LacaNeshveyev:JFA04, Thomsen:AM2017}) brought them to the fore.
It has been known for some time \cite{KatsoulisKribs:MA04, Solel:JAMS04} that the
\emph{non-selfadjoint} Toeplitz algebra (also called the tensor algebra or the quiver operator
algebra) of a directed graph $E$ contains all of the information about $E$---if $E$ and $F$ are
directed graphs with isomorphic tensor algebras, then they are themselves isomorphic. But there
are no results in this direction for the Toeplitz $C^*$-algebras of directed graphs.

Here we consider the extent to which a finite directed graph can be recovered from its Toeplitz
algebra and gauge action. We show that at least one additional piece of information is needed (see
Examples \ref{ex:not iso}~and~\ref{eg:sinks}) and identify two pieces of information, either of
which suffices for finite graphs with no sinks. Our key tool is the KMS structure of $\Tt C^*(E)$
for the dynamics arising from its gauge action; we show that using this we can recover the
rank-one projections in $\Tt C^*(E)$ that correspond to the vertices of $E$. From this, using the
spectral subspaces of the gauge action, it is straightforward to count the number of edges
(indeed, the number of paths of length $n$ for any $n$) emanating from a given vertex. However,
additional information is required to determine which of these paths have the same ranges. We show
that the subalgebra $M_E = \lsp\{q_v : v \in E^0\} \subseteq \Tt C^*(E)$ generated by the vertex
projections is sufficient to recover this information, and that  if $E$ has no sinks then the
action $\kappa^E$ of the torus $\TT^{E^0}$ such that $\kappa^{{E}}_z(t_e) = z_{s(e)} t_e$ for each
$e \in E^1$ also suffices.

\smallskip

\textbf{Conventions.} We use the conventions of \cite{Raeburn} for graphs and their
$C^*$-algebras; so the Toeplitz algebra of a directed graph $E$ is the universal $C^*$-algebra
generated by projections $\{q_v : v \in E^0\}$ and partial isometries $\{t_e : e \in E^1\}$ such
that $t_e^* t_e = q_{s(e)}$ and $q_v \ge \sum_{r(e) = v} t_e t^*_e$. We use the notational
convention in which, for example, $vE^1 = r^{-1}(v)$ and $E^1 v = s^{-1}(v)$.

\section{An example}

We started this project by asking whether it is possible to recover a directed graph $E$ from its
Toeplitz algebra and gauge action. The following example shows that the answer is no, even for the
particularly well-behaved class of strongly connected finite graphs in which every cycle has an
entrance. We thank S{\o}ren Eilers for very helpful conversations that led to the construction of
this example. For a simpler example involving graphs that are not strongly connected, and have
sinks and sources, see Example~\ref{eg:sinks}.

\begin{exm}\label{ex:not iso}
Consider the following directed graphs $E$ and $F$, that differ only in the range of the edge $e$:
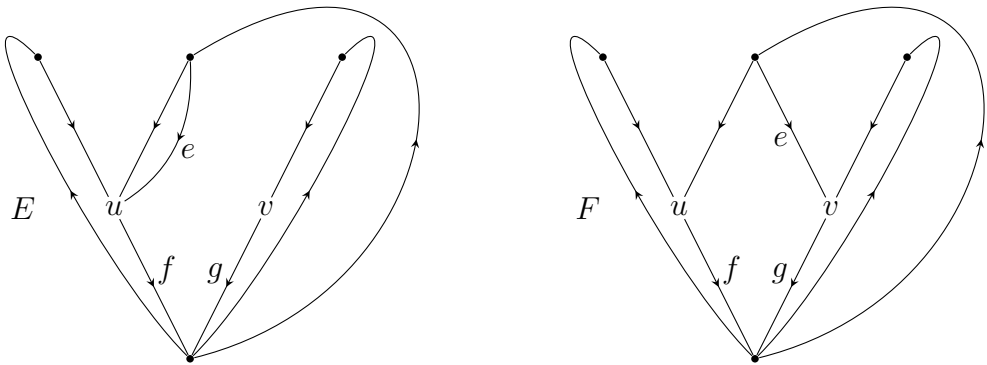
\[
\tikzset{->-/.default=0.5, ->-/.style={decoration={
  markings,
  mark=at position #1 with {\arrow{stealth}}},postaction={decorate}}}
\vbox to 5cm{\vskip-1.6cm\begin{tikzpicture}
    \node at (-2.2, 2) {$E$};
    \node[circle, inner sep=1pt, fill=black] (a) at (0,0) {};
    \node[circle, inner sep=0.2pt] (b) at (-1,2) {$u$};
    \node[circle, inner sep=0.2pt] (c) at (1,2) {$v$};
    \node[circle, inner sep=1pt, fill=black] (d) at (-2,4) {};
    \node[circle, inner sep=1pt, fill=black] (e) at (0,4) {};
    \node[circle, inner sep=1pt, fill=black] (f) at (2,4) {};
    \draw[->-] (b)--(a) node [pos=0.5, anchor=south west, inner sep=1pt] {$f$};
    \draw[->-] (c)--(a) node [pos=0.5, anchor=south east, inner sep=1pt] {$g$};
    \draw[->-] (d)--(b);
    \draw[->-] (e) to (b);
    \draw[->-, in=33, out= 273] (e) to node [pos=0.5, anchor=north west, inner sep=1pt] {$e$} (b);
    \draw[->-] (f)--(c);
    \draw[->-, in=135, out=135] (a) to (d);
    \draw[->-] (a)  .. controls +(4, 1) and +(4,2.5) .. (e);
    \draw[->-, in=45, out=45] (a) to (f);
\end{tikzpicture}
\begin{tikzpicture}
    \node at (-2.2, 2) {$F$};
    \node[circle, inner sep=1pt, fill=black] (a) at (0,0) {};
    \node[circle, inner sep=0.2pt] (b) at (-1,2) {$u$};
    \node[circle, inner sep=0.2pt] (c) at (1,2) {$v$};
    \node[circle, inner sep=1pt, fill=black] (d) at (-2,4) {};
    \node[circle, inner sep=1pt, fill=black] (e) at (0,4) {};
    \node[circle, inner sep=1pt, fill=black] (f) at (2,4) {};
    \draw[->-] (b)--(a) node [pos=0.5, anchor=south west, inner sep=1pt] {$f$};
    \draw[->-] (c)--(a) node [pos=0.5, anchor=south east, inner sep=1pt] {$g$};
    \draw[->-] (d)--(b);
    \draw[->-] (e) to (b);
    \draw[->-] (e) to node [pos=0.5, anchor=north east, inner sep=1pt] {$e$} (c);
    \draw[->-] (f)--(c);
    \draw[->-, in=135, out=135] (a) to (d);
    \draw[->-] (a)  .. controls +(4, 1) and +(4,2.5) .. (e);
    \draw[->-, in=45, out=45] (a) to (f);
\end{tikzpicture}}
\]
Let $(t, q)$ be the universal generating Toeplitz--Cuntz--Krieger $F$-family in $\Tt
C^*(F)$. Define elements $\{Q_w : w \in E^0\}$ and $\{T_h : h \in E^1\}$ in $\Tt C^*(F)$ as follows:
\begin{gather*}
Q_u = q_u + t_e t^*_e,\quad Q_v = q_v - t_et_e^*,\quad
T_f = t_f + t_g t_e t^*_e,\quad T_g = t_g (q_v - t_et_e^*),\quad\text{and}\\
Q_w = q_w \text{ for $w \in E^0 \setminus \{u,v\}$}\quad\text{ and }\quad T_h = t_h
\text{ for $h \in E^1 \setminus \{f, g\}$.}
\end{gather*}
It is routine to check that $(Q, T)$ is a Toeplitz--Cuntz--Krieger $E$-family that generates $\Tt
C^*(F)$, and that the elements $Q_w - \sum_{h \in wE^1} T_h T^*_h$ are all nonzero. So the
universal property of $\Tt C^*(E)$ yields a surjective homomorphism $\pi_{Q,T} : \Tt C^*(E) \to
\Tt C^*(F)$ such that $\pi_{Q,T}(q_w) = Q_w$ and $\pi_{Q,T}(t_h) = T_h$, and
\cite[Theorem~4.1]{FR} implies that $\pi_{Q,T}$ is injective. It is immediate from the definitions
of the $T_h$ and $Q_w$ that $\pi_{Q,T}$ is gauge-equivariant. So $(\Tt C^*(E), \gamma^E) \cong
(\Tt C^*(F), \gamma^F)$, but there is no graph-isomorphism from $E$ to $F$ because, for example,
$E$ has a pair of parallel edges, whereas $F$ does not.

In fact, since the canonical diagonals $D_E = \clsp\{t_\mu t^*_\mu : \mu \in E^*\}$ and $D_F =
\clsp\{t_\mu t^*_\mu : \mu \in F^*\}$ are maximal abelian in $\Tt C^*(E)$ and $\Tt
C^*(F)$, 
we see that $\pi_{Q,T}(D_E)$ is a maximal abelian subalgebra of $\Tt C^*(F)$.  Since this maximal
abelian subalgebra is contained in the maximal abelian subalgebra $D_F$ of $\Tt C^*(F)$, we deduce
that $\pi_{Q,T}(D_E) = D_F$. So the triples $(\Tt C^*(E), \gamma^E, D_E)$ and $(\Tt C^*(F),
\gamma^F, D_F)$ are isomorphic even though $E$ and $F$ are not.
\end{exm}

\section{The main theorem}

Example~\ref{ex:not iso} shows that recovering a directed graph from its Toeplitz algebra requires
more information than just the gauge action. Our main result identifies two additional bits of
data, either one of which bridges the gap. The first one is the C*-subalgebra generated by the
vertex projections inside the Toeplitz algebra. The second one is a higher dimensional
generalisation of the gauge action.

\begin{dfn}\label{dfn:kappa}
When $E$ is a directed graph, the {\em generalised gauge action} on $\Tt C^*(E)$ is the action
$\kappa^E : \TT^{E^0} \to \Aut \Tt C^*(E)$  determined by $\kappa^E_z(t_e) = z_{s(e)} t_e$ for all
$e \in E^1$ and $z\in \TT^{E^0}$. When $E$ and $F$ are two directed graphs, we say that an
isomorphism $\rho : \Tt C^*(E) \to \Tt C^*(F)$ {\em intertwines the generalised gauge actions}
$\kappa^E$ and $\kappa^F$ if there is a bijection $\varphi : E^0 \to F^0$ such that the induced
homomorphism $\varphi^*: \TT^{E^0} \to \TT^{F^0}$ satisfies  $\rho \circ \kappa^E_z =
\kappa^F_{\varphi^*(z)} \circ \rho$ for all $z \in \TT^{E^0}$.
\end{dfn}

\begin{thm}\label{thm:main}
Let $E$ and $F$ be finite directed graphs. As before, let $\gamma^E$ be the gauge action of $\TT$
on $\Tt C^*(E)$, and let $M_E := \lsp\{q_v : v \in E^0\} \subseteq \Tt C^*(E)$. Let $\kappa^E $
be the generalised gauge action of $ \TT^{E^0}$ on $ \Tt C^*(E)$ given by $\kappa^E_z(t_e) =
z_{s(e)} t_e$ for all $e \in E^1$. Denote by $\gamma^F$, $M_F$, and $\kappa^{{F}}$  the
corresponding concepts for  $\Tt C^*(F)$. \begin{enumerate}
    \item\label{it:M} There is an isomorphism $\Tt C^*(E) \cong \Tt C^*(F)$ that intertwines
        $\gamma^E$ and $\gamma^F$ and carries $M_E$ to $M_F$ if and only if $E \cong F$.
    \item\label{it:kappa} Suppose that $E$ and $F$ have no sinks. Then there is an isomorphism
        $\Tt C^*(E) \cong \Tt C^*(F)$ that intertwines the generalised gauge actions $\kappa^E$ and
        $\kappa^F$ if and only if $E \cong F$.
\end{enumerate}
\end{thm}

\begin{rmk}
In both parts of Theorem~\ref{thm:main}, the additional data required beyond the gauge actions
includes the number of vertices in the graphs. We point out, however, that this number is already
available as an isomorphism invariant of the $C^*$-algebra $\Tt C^*(E)$ alone: by
\cite[Theorem~4.1]{FR} combined with \cite[Theorem~4.4]{Pimsner}, the Toeplitz algebra $\Tt
C^*(E)$ is $KK$-equivalent to $\CC^{E^0}$, and in particular $K_0(\Tt C^*(E)) \cong \ZZ E^0$. So
if $\Tt C^*(E) \cong \Tt C^*(F)$, we already know that $|E^0| = |F^0|$.
\end{rmk}

The proof of the ``if" implication is easy in both cases.  If $\varphi^0 : E^0 \to F^0$ and
$\varphi^1 : E^1 \to F^1$ constitute an isomorphism of graphs, then the isomorphism $\rho : \Tt
C^*(E) \to \Tt C^*(F)$ given by $\rho(t_e) = t_{\varphi^1(e)}$ and $\rho(q_v) = q_{\varphi^0(v)}$
carries $M^E$ to $M^F$
 and intertwines $\kappa^E$ and $\kappa^F$ (and, by restriction, $\gamma^E$ and $\gamma^F$), via the isomorphism $\TT^{E^0} \cong \TT^{F^0}$ induced by
$\varphi^0$.

To prove the reverse implications we shall use the results of \cite{aHLRS, aHLRS3} on the KMS
structure of the Toeplitz algebra $\Tt C^*(E)$ for the dynamics $\alpha:\RR \to \Aut(\Tt C^*(E))$,
where $\alpha_t = \gamma^E_{e^{it}}$ is the lift of the gauge action; that is
\begin{equation}\label{eq:alpha def}
    \alpha_t(q_v) = q_v\quad\text{ and }\quad \alpha_t(t_e) = e^{it} t_e\quad\text{ for all }v \in E^0, e \in E^1,\text{ and }t \in \RR.
\end{equation}
We write
\[
\Ext_\beta(\alpha) := \{\phi : \phi\text{ is an extremal KMS$_\beta$ state of $(\Tt C^*(E), \alpha)$}\}.
\]

We first need to be able to recognise, using the data $(\Tt C^*(E), \alpha)$, when a real number
$\beta$ is strictly greater than the natural logarithm of the spectral radius of the adjacency
matrix $A_E$ of the directed graph $E$. For this, as in \cite{aHLRS3}, we write $\sim$ for the
equivalence relation on $E^0$ given by $v \sim w$ if both $v E^* w \not= \emptyset$ and $w E^* v
\not= \emptyset$. We call the equivalence classes $C \in E^0/{\sim}$ the \emph{strongly connected
components} of $E$. For $C \in E^0/{\sim}$, we write $A_C$ for the $C \times C$ submatrix of
$A_E$, which is the adjacency matrix of the subgraph of $E$ with vertices $C$ and edges $C E^1 C$.

\begin{lem}\label{lem:critical beta}
Let $E$ be a finite directed graph. If $\beta < \log\rho(A_E) < \beta'$, then
$|\Ext_\beta(\alpha)| < |\Ext_{\beta'}(\alpha)|$.
\end{lem}
\begin{proof}
If $E$ has no cycles, then \cite[Lemma~A.1(b)]{aHLRS} shows that $\log \rho(A_E) = -\infty$, and
so the result is vacuous. So suppose that $E$ has at least one cycle. Then $\rho(A_E) =
\max\{\rho(A_C) : C\text{ is a nontrivial strongly connected component of $E$}\}$, as discussed at
the beginning of \cite[Section~4]{aHLRS3}. Let $H_\beta := \{s(\mu) : \mu \in E^*\text{ and
}r(\mu) \in \bigcup_{\log\rho(A_C) > \beta} C\}$. Theorem~3.1 of \cite{aHLRS} shows that
$|\Ext_{\beta'}(\alpha)| = |E^0|$, and Theorem~5.3 of \cite{aHLRS3} implies that
$|\Ext_{\beta}(\alpha)| \le |E^0 \setminus H_\beta|$. Since $\beta < \log\rho(A_E) =
\max\{\log\rho(A_C) : C \in E^0/{\sim}\}$, we have $H_\beta \not= \emptyset$. Hence $|E^0
\setminus H_\beta| < |E^0|$, which proves the result.
\end{proof}

{\begin{lem}\label{cor:beta property}
The interior in $\RR$ of the set
\begin{equation}\label{eq:half-line}
\big\{\beta \in (0,\infty) : |\Ext_{\beta'}(\alpha)| = |\Ext_{\beta}(\alpha)|\text{ for all }\beta' \ge \beta\big\}
\end{equation}
is the open half-line $\big(\log\rho(A_E), \infty\big)$.
\end{lem}}
\begin{proof}
Theorem~3.1 of \cite{aHLRS} shows that if $\beta > \log\rho(A_E)$, then we have
$|\Ext_{\beta'}(\alpha)| = |\Ext_{\beta}(\alpha)|\text{ for all }\beta' \ge \beta$, and
Lemma~\ref{lem:critical beta} shows that if $\beta < \log\rho(A_E)$, then $|\Ext_{\beta'}(\alpha)|
> |\Ext_{\beta}(\alpha)|\text{ for some }\beta' > \beta$.
\end{proof}

{Throughout the rest of this note, we shall let $\pi : \Tt C^*(E) \to \Bb(\ell^2(E^*))$ be the canonical
(faithful) path-space representation of $\Tt C^*(E)$. We will need to show that
the minimal projections in $\Tt C^*(E)$ corresponding to vertices of $E$ can be recovered using
the gauge action $\gamma^E$. For each $\mu \in E^*$, we define
\begin{equation}\label{oldnotation4}\textstyle
\Delta_\mu := t_\mu \big(q_{s(\mu)} - \sum_{e \in s(\mu) E^1} t_e t^*_e\big) t^*_\mu \in \Tt C^*(E).
\end{equation}
The $\Delta_\mu$ are minimal projections in the canonical copy of $\bigoplus_{v \in E^0}
\Kk(\ell^2(E^* v))$ in $\Tt C^*(E)$; indeed, each $\pi(\Delta_\mu)$ is the rank-1 projection
$\theta_{\delta_\mu, \delta_\mu}$ onto the span of the basis vector $\delta_\mu \in \ell^2(E^*)$.
}

\begin{lem}\label{lem:get vertices}
Let $E$ be a finite directed graph. Let $\alpha$ be the dynamics~\eqref{eq:alpha def}. Let $\beta$
be any real number greater than $\max\{0, \log\rho(A_E)\}$. Let $\phi$ be an extremal KMS$_\beta$
state of $(\Tt C^*(E), \alpha)$. Let $P_{\min}$ denote the collection of minimal projections on
$\Tt C^*(E)$. There is a unique $p_\phi \in P_{\min}$ such that $\phi(p_\phi) = \max\{\phi(q) : q
\in P_{\min}\}$. Moreover, with $\Delta_v$ as in \eqref{oldnotation4}, we have $p_\phi =
\Delta_{v_\phi}$ for some $v_\phi \in E^0$.
\end{lem}
\begin{proof}
For each $v\in E^0$, let $\varepsilon^v_{(\cdot)}$ denote the measure $\big(\sum_{\mu \in E^* v}
e^{-\beta|\mu|}\big)^{-1} \delta_v {(\cdot)}$  on $E^0$. Since $\beta > \log\rho(A_E)$,
\cite[Theorem~3.1]{aHLRS} implies that there is a unique $v_\phi \in E^0$ such that $\phi$
satisfies
\[
\phi(t_\mu t^*_\nu) = \delta_{\mu,\nu} e^{-\beta|\mu|} \varepsilon^{v_\phi}_{s(\mu)},\quad \text{for all $\mu,\nu\in E^*$}.
\]
By the proof of \cite[Theorem~3.1(b)]{aHLRS}, we know that $\phi$ satisfies
\[
\phi(a)
    = \sum_{\mu \in E^*v_\phi} e^{-\beta|\mu|} \big(\pi(a) \delta_\mu | \delta_\mu\big) \varepsilon^{v_\phi}_{s(\mu)}\quad \text{for all $a\in \Tt C^*(E)$}.
\]
We have
\[\textstyle
\phi(\Delta_{v_\phi}) = \varepsilon^{v_\phi}_{v_\phi} = \big(\sum_{\mu \in E^*v_\phi} e^{-\beta|\mu|}\big)^{-1}.
\]
Fix $q \in P_{\min} \setminus \Delta_{v_\phi}$. It suffices to show that $\phi(q) <
\phi(\Delta_{v_\phi})$. Let $\pi_{v_\phi} : \Tt C^*(E) \to \Bb(\ell^2(E^* v_\phi))$ be the direct
summand in $\pi$ corresponding to ${v_\phi}$. Then $\phi$ factors through $\pi_{v_\phi}$. If
$\phi(q) = 0$ then we certainly have $\phi(q) < \phi(\Delta_{v_\phi})$, so suppose that $\phi(q)
\not= 0$. Then $\pi_{v_\phi}(q) \not= 0$, and so $\pi_{v_\phi}(q)$ is a minimal projection in
$\pi_{v_\phi}(\Tt C^*(E))$. Since $\pi_{v_\phi}(\Tt C^*(E))$ contains all of
$\Kk(\ell^2(E^*{v_\phi}))$, it follows that $\pi(q)$ is the rank-one projection $\theta_{\xi,\xi
}$projection corresponding to a unit
vector $\xi \in \ell^2(E^*{v_\phi})$. 
Hence
\begin{align*}
\phi(q) &= \sum_{\mu \in E^*{v_\phi}} e^{-\beta|\mu|} \big(\pi(q) \delta_\mu | \delta_\mu\big) \varepsilon^{v_\phi}_{s(\mu)}
     = \sum_{\mu \in E^*{v_\phi}} e^{-\beta|\mu|} \big(\theta_{\xi,\xi}(\delta_\mu) | \delta_\mu\big) \varepsilon^{v_\phi}_{s(\mu)} \\
    &= \sum_{\mu \in E^*{v_\phi}} e^{-\beta|\mu|} \big(\big(\xi \mid \delta_\mu\big)\xi | \delta_\mu\big) \varepsilon^{v_\phi}_{s(\mu)}.
    \end{align*}
    Since $\beta > 0$, we have $e^{-\beta |\mu |} =1$ when
$\mu = v_\phi$ and $e^{-\beta |\mu | } \leq e^{-\beta} $
when $\mu \neq v_\phi$, and so we deduce that
   \[
   \phi(q)
    \le \Big(|\xi_{v_\phi}|^2 + e^{-\beta} \sum_{\mu \not= {v_\phi}} |\xi_\mu|^2\Big) \varepsilon^{v_\phi}_{s(\mu)}.
   \]

Since $q \not= \Delta_{v_\phi}$, we have $\xi \not= \delta_{v_\phi}$, and so $|\xi_{v_\phi}| < 1$.
We have $\sum |\xi_\mu|^2 = \|\xi\|^2 = 1$, and so $e^{-\beta} \sum_{\mu \not= v_\phi} |\xi_\mu|^2
= e^{-\beta} (1 - |\xi_{v_\phi}|^2) < 1 - |\xi_{v_\phi}|^2$. Hence $\phi(q) <
\varepsilon^v_{s(\mu)} = \phi(\Delta_{v_\phi})$ as claimed.
\end{proof}

Lemma~\ref{lem:get vertices} allows us to recover the projections $\Delta_v$ of $\Tt C^*(E)$ from
$\Tt C^*(E)$ together with its simplex of KMS states. Since the KMS states are intrinsic to the
pair $(\Tt C^*(E), \gamma^E)$, it follows that we can recover the $\Delta_v$ from the Toeplitz
algebra and its gauge action. We  show next how to recover the cardinalities of the sets $E^n v$
as well. We start with some notation.

\begin{ntn}
For each $\mu,\nu \in E^*$ with $s(\mu)=s(\nu)$ we define $\Theta_{\mu,\nu} := t_\mu
\Delta_{s(\mu)} t^*_\nu$. Recall that the path-space representation $\pi$ carries each
$\Theta_{\mu,\nu}$ to the canonical matrix unit $\theta_{\delta_\mu, \delta_\nu}$. Recall also
that for $n \in \ZZ$, the \emph{$n$\textsuperscript{th} spectral subspace} $\Tt C^*(E)_n$ of $\Tt
C^*(E)$ with respect to $\gamma$ is
\[
    \Tt C^*(E)_n := \{a \in \Tt C^*(E) : \gamma^E_z(a) = z^n a\text{ for all }z \in \TT\}.
\]
\end{ntn}

\begin{lem}\label{lem:get edges}
Let $E$ be a finite directed graph. For $n \ge 0$, we have $\Tt C^*(E)_n \Delta_v =
\lsp\{\Theta_{\mu, v} : \mu \in E^n v\}$; in particular, $|E^n v| = \dim(\Tt C^*(E)_n \Delta_v)$.
\end{lem}
\begin{proof}
It is standard that $\Tt C^*(E)_n = \clsp\{t_\mu t^*_\nu : \mu,\nu \in E^*, |\mu| - |\nu| = n,
s(\mu) = s(\nu)\}$. The path-space representation $\pi$ carries $\Delta_v$ to $\theta_{\delta_v,
\delta_v}$, and carries each $t_\mu t^*_\nu$ to the strong-operator sum $\sum_{\lambda \in
s(\nu)E^*} \theta_{\delta_{\mu\lambda},\delta_{\nu\lambda}}$. The latter is nonzero at $\delta_v$
only if $v = \nu\lambda$ for some $\lambda \in s(\mu)E^*$, which forces $\nu = v = \lambda =
s(\mu)$. So if $a \in \Tt C^*(E)_n$ and $a \Delta_v \not= 0$, then $a \Delta_v \in \lsp\{t_\mu
\Delta_v : \mu \in E^n v\} = \lsp\{\Theta_{\mu,v} : \mu \in E^n v\}$. Since each $\Theta_{\mu, v}
= \Theta_{\mu, v} \Delta_v$, the reverse containment is clear.
\end{proof}

We can now prove the first part of the main theorem.

\begin{proof}[Proof of Theorem~\ref{thm:main}(\ref{it:M})]
{By Lemma~\ref{cor:beta property} we may determine the value of $\log\rho(A_E)$ from the KMS state structure of $\alpha$,} and then choose $\beta > \log\rho(A_E)$. For
$\phi \in \Ext_\beta(\alpha)$, Lemma~\ref{lem:get vertices} yields a unique minimal projection
$p_\phi$ of $\Tt C^*(E)$ such that $\phi(p_\phi) = \max\{\phi(q) : q\text{ is a minimal projection
of }\Tt C^*(E)\}$, and we have $p_\phi = \Delta_{v_\phi}$ for some $v_\phi \in E^0$. We have
$q_{v_\phi} \ge \Delta_{v_\phi}$, and then for $w \not= v_\phi$ in $E^0$ we have $q_w
\Delta_{v_\phi} = q_w q_{v_\phi} \Delta_{v_\phi} = 0$. So there is a unique minimal projection
$P_\phi \in M_E$ that dominates $p_\phi$, namely $P_\phi = q_{v_\phi}$.

For $\phi,\psi \in \Ext_\beta(\alpha)$, let
\[
N(\phi,\psi) := \dim P_\phi \Tt C^*(E)_1 p_\psi.
\]
Let $\widetilde{E}$ be the directed graph with vertices $\Ext_\beta(\alpha)$ and with $|\phi
\widetilde{E}^1\psi| = N(\phi,\psi)$ for all $\phi,\psi \in \Ext_\beta(\alpha)$.
By construction, the graph $\widetilde{E}$ is an isomorphism invariant of the triple $(\Tt C^*(E), \gamma^E, M_E)$.
We claim that $\widetilde{E} \cong E$.

We know from Lemma~\ref{lem:get vertices} that $\phi \mapsto v_\phi$ from $\widetilde{E}^0$ to
$E^0$ is a bijection, so it suffices to show that $N(\phi,\psi) = |v_\phi E^1 v_\psi|$ for all
$\phi,\psi$. Lemma~\ref{lem:get edges} shows that $\Tt C^*(E)_1 p_\psi = \lsp\{\Theta_{e,v_\psi} :
e \in E^1 v_\psi\}$. Since for each $e \in E^1 v_\psi$ we have $\Theta_{e,v_\psi}
\Theta_{e,v_\psi}^* = \Theta_{e,e} \le q_{r(e)}$, it follows that $P_\phi \Tt C^*(E)_1 p_\psi =
\lsp\{\Theta_{e,v_\psi} : e \in v_\phi E^1 v_\psi\}$. Hence
\[
|v_\phi E^1 v_\psi| = \dim P_\phi \Tt C^*(E)_1 p_\psi = N(\phi,\psi).
\]

So $\widetilde{E} \cong E$, as claimed. Applying the process of the preceding three paragraphs to
the system $(\Tt C^*(F), \gamma^F, M_F)$ we obtain a graph $\widetilde{F} \cong F$. Since the
systems $(\Tt C^*(E), \gamma^E, M_E)$ and $(\Tt C^*(F), \gamma^F, M_F)$ are isomorphic, we see
that $\widetilde{E} \cong \widetilde{F}$, and therefore $E \cong F$.
\end{proof}

To prove statement~(\ref{it:kappa}) of Theorem~\ref{thm:main} we first show how to determine which
coordinate of the generalised gauge action $\kappa^E$ corresponds to the minimal projection
$p_\phi$ obtained from $\phi \in \Ext_{\beta}(\alpha)$ as in Lemma~\ref{lem:get vertices}.

\begin{lem}\label{lem:identify vertices}
Let $E$ be a finite directed graph with no sinks, and let $\kappa^E$ and $\alpha$ be as in
Definition~\ref{dfn:kappa} and~\eqref{eq:alpha def}. Fix $\beta > \ln\rho(A_E)$ and let $\phi$ be
an extremal KMS$_\beta$ state of $(\Tt C^*(E), \alpha)$. Let $p_\phi$ be the projection of
Lemma~\ref{lem:get vertices}. Then the vertex $v_\phi$ such that $p_\phi = \Delta_{v_\phi}$ is the
unique vertex such that $\kappa^E_z(a) = z_{v_\phi} a$ for all $a \in \Tt C^*(E)_1 p_\phi$ and
$z\in \TT^{E^0}$.
\end{lem}
\begin{proof}
For $w \in E^0$, Lemma~\ref{lem:get edges} gives $\Tt C^*(E)_1\Delta_w = \lsp\{\Theta_{e, w} : e
\in E^1 w\} = \lsp\{t_e \Delta_w : e \in E^1\}$, and so it follows from the definition of
$\kappa^E$ that $\kappa^E_z(a) = z_w a$ for all $a \in \Tt C^*(E)_1\Delta_w$ and $z \in
\TT^{E^0}$. Since $E$ has no sinks, each $\lsp\{\Theta_{e, w} : e \in E^1 w\}$ is nontrivial,
which proves uniqueness.
\end{proof}

\begin{proof}[Proof of Theorem~\ref{thm:main}(\ref{it:kappa})]
First observe that the dynamics $\alpha$ of $\Tt C^*(E)$ defined in~\eqref{eq:alpha def} is
determined by $\kappa^E$ via  $\alpha_t = \kappa^E_{(e^{it}, \dots, e^{it})}$. Using
Lemma~\ref{cor:beta property} as in the proof of Theorem~\ref{thm:main}(\ref{it:M}), fix $\beta >
\ln\rho(A_E)$. For each extremal KMS$_\beta$ state $\phi \in \Ext_\beta(\alpha)$,
Lemma~\ref{lem:get vertices} yields a unique minimal projection $p_\phi$ of $\Tt C^*(E)$ such that
\[\phi(p_\phi) = \max\{\phi(q) : q\text{ is a minimal projection of } \Tt C^*(E)\}.\]
Lemma~\ref{lem:identify vertices} shows that $p_\phi = \Delta_{v_\phi}$ where $v_\phi \in E^0$ is
the unique vertex such that $\kappa^E_z(a) = z_{v_\phi} a$ for all $a \in \Tt C^*(E)_1 p_\phi$.

Suppose that $\phi, \psi \in \Ext_\beta(\alpha)$ are distinct. For $z \in \TT$ let
$\omega(\phi,\psi,z) \in \TT^{E^0}$ be the element such that
\[
\omega(\phi,\psi,z)_u = \begin{cases}
    \overline{z} &\text{ if $u = v_\phi$}\\
    z &\text{ if $u = v_\psi$}\\
    1 &\text{ otherwise.}
    \end{cases}
\]
Define an action $\gamma^{\phi,\psi} : \TT \to \Aut(\Tt C^*(E))$ by $\gamma^{\phi,\psi}_z =
\kappa^E_{\omega(\phi,\psi,z)}$. Note that this action fixes the partial isometry $t_{ef}$
associated to $ef \in E^2 v_\psi$ if and only if $r(f) = s(e) = v_\phi$. Combining the fixed point
algebra $\Tt C^*(E)^{\gamma^{\phi,\psi}}$ of $\gamma^{\phi,\psi}$ with the second spectral
subspace of the gauge action $\gamma^E$, we define
\[
N(\phi,\psi) :=
        \dim\big(\Tt C^*(E)^{\gamma^{\phi,\psi}} \cap \Tt C^*(E)_2 p_\psi\big) / \dim(\Tt C^*(E)_1 p_\phi).
\]
We extend the definition of $N$ to the case $\phi = \psi \in \Ext_\beta(\alpha)$ by setting
\[
 N(\psi, \psi) := \dim(\Tt C^*(E)_1 p_\psi) - \sum_{\phi \not= \psi} N(\phi,\psi).
\]

We claim that $N(\phi,\psi)\in \NN$ for all $\phi,\psi \in \Ext_\beta(\alpha)$, and that $E$ is
isomorphic to the directed graph $\widetilde{E}$ with vertices $\widetilde{E}^0 :=
\Ext_\beta(\alpha)$, and such that $|\phi \widetilde{E}^1 \psi| = N(\phi, \psi)$ for all
$\phi,\psi \in \Ext_\beta(\alpha)$.
Since we already have a bijection $\phi \mapsto v_\phi$ from $\widetilde{E}^0$ to $E^0$, to prove
the claim, we just have to show that $N(\phi,\psi) = |v_\phi E^1 v_\psi|$ for all $\phi,\psi$.

For this, fix $\phi,\psi \in \Ext_\beta(\alpha)$ and let $ ef \in E^2$. Then
\[
\gamma^{\phi,\psi}_z(t_{ef} p_\psi) = \begin{cases}
    t_{ef} p_\psi &\text{ if $f \in v_\phi E^1 v_\psi$}\\
    z^2 t_{ef} p_\psi &\text{ if $f \in v_\psi E^1 v_\psi$}\\
    zt_{ef} p_\psi &\text{ if $f \in E^1v_\psi \setminus (v_\phi E^1 v_\psi  \cup v_\psi E^1 v_\psi)$}\\
    0 &\text{ if $f \not\in E^1 v_\psi$.}
    \end{cases}
\]
So Lemma~\ref{lem:get edges} implies that $\Tt C^*(E)^{\gamma^{\phi,\psi}} \cap \Tt C^*(E)_2
p_\psi = \lsp\{\Theta_{ef, v_\psi} : ef \in E^1 v_\phi E^1 v_\psi\}$. Hence,  $|E^1 v_\phi| \cdot
|v_\phi E^1 v_\psi| = |E^1 v_\phi E^1 v_\psi| = \dim(\Tt C^*(E)^{\gamma^{\phi,\psi}} \cap \Tt
C^*(E)_2 p_\psi)$. By Lemma~\ref{lem:get edges}, we have $|E^1 v_\phi| = \dim(\Tt C^*(E)_1
p_\phi)$. Since, by hypothesis, $E$ has no sinks, we have $|E^1 v_\phi| \not= 0$, and so we deduce
that
\[
|v_\phi E^1 v_\psi| = \dim\big(\Tt C^*(E)^{\gamma^{\phi,\psi}} \cap \Tt C^*(E)_2 p_\psi\big) / \dim(\Tt C^*(E)_1 p_\phi)
    = N(\phi, \psi).
\]
Now for each $\psi \in \Ext_\beta(\alpha)$, we see that
\begin{align*}
|v_\psi E^1 v_\psi| &= |E^1 v_\psi| - \sum_{\phi\not= \psi} |v_\phi E^1 v_\psi|\\
    &= \dim(\Tt C^*(E)_1 p_\psi) - \sum_{\phi \not= \psi} \dim\big(\Tt C^*(E)^{\gamma^{\phi,\psi}} \cap \Tt C^*(E)_2 p_\psi\big) / \dim(\Tt C^*(E)_1 p_\phi)\\
    &= N(\psi, \psi).
    \end{align*}
{This shows that $E\cong \widetilde E$ and concludes the proof of the claim.}

{
To finish the proof of the ``only if" assertion in
Theorem~\ref{thm:main}(\ref{it:kappa}) assume now there exist
an isomorphism $\rho: \Tt C^*(E) \to \Tt C^*(F)$ and a bijection $\varphi:E^0 \to F^0$
intertwining the generalised gauge actions $\kappa^E$ and $\kappa^F$. Then
$\varphi^*: \TT^{E^0} \to \TT^{F^0}$ maps constant functions to constant functions,
that is, $\varphi^*$ respects the diagonal  embeddings of $\TT$.
Hence $\rho$ intertwines the gauge actions $\gamma^E$ and $\gamma^F$, and also the dynamics $\alpha^E$ and $\alpha^F$ obtained from them on setting $z = e^{it}$.
Passing to extremal KMS$_\beta$ states,  we get a bijection $\widetilde{E}^0 :=\Ext_\beta(\alpha^E) \cong
\Ext_\beta(\alpha^F) =: \widetilde{F}^0$ in which $\phi \mapsto \phi' := \phi\circ \rho\inv$. The isomorphism $\rho$ also intertwines the action $\gamma^{\phi,\psi} : \TT \to \Aut(\Tt C^*(E))$ with the action $\gamma^{\phi',\psi'} : \TT \to \Aut(\Tt C^*(F))$  and thus $N(\phi,\psi) = N(\phi',\psi')$. Thus, much like
in the final paragraph of the proof of the ``only if" assertion in Theorem~\ref{thm:main}(\ref{it:M}), we conclude that $\widetilde E \cong \widetilde F$ and hence that $E\cong F$.}
\end{proof}

\begin{exm}\label{eg:sinks}
As compared to statement~(\ref{it:M}), statement~(\ref{it:kappa}) of our main theorem has the
additional hypothesis that $E$ and $F$ have no sinks. Here we present an example---first shown to
the fourth author in the context of Cohn path algebras by Gene Abrams, and then independently by
S{\o}ren Eilers---that shows that the additional hypothesis in statement~(\ref{it:kappa}) is
necessary. Consider the graphs
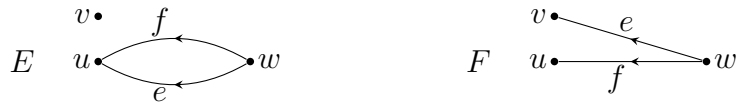
\[
\tikzset{->-/.default=0.5, ->-/.style={decoration={
  markings,
  mark=at position #1 with {\arrow{stealth}}},postaction={decorate}}}
\begin{tikzpicture}[scale=2]
    \node at (0,0) {$E$};
    \node[circle, inner sep=1pt, fill=black] (u) at (0.5,0) {};
    \node[circle, inner sep=1pt, fill=black] (v) at (0.5,0.3) {};
    \node[circle, inner sep=1pt, fill=black] (w) at (1.5,0) {};
    \node[inner sep=1pt, anchor=east] at (u.west) {$u$};
    \node[inner sep=1pt, anchor=east] at (v.west) {$v$};
    \node[inner sep=1pt, anchor=west] at (w.east) {$w$};
    \draw[->-, out=210, in=330] (w) to node [pos=0.6, anchor=north, inner sep=1pt] {$e$} (u);
    \draw[->-, out=150, in=30] (w) to node [pos=0.6, anchor=south, inner sep=1pt] {$f$} (u);
    \node at (3,0) {$F$};
    \node[circle, inner sep=1pt, fill=black] (u2) at (3.5,0) {};
    \node[circle, inner sep=1pt, fill=black] (v2) at (3.5,0.3) {};
    \node[circle, inner sep=1pt, fill=black] (w2) at (4.5,0) {};
    \node[inner sep=1pt, anchor=east] at (u2.west) {$u$};
    \node[inner sep=1pt, anchor=east] at (v2.west) {$v$};
    \node[inner sep=1pt, anchor=west] at (w2.east) {$w$};
    \draw[->-] (w2) to node [pos=0.6, anchor=south west, inner sep=1pt] {$e$} (v2);
    \draw[->-] (w2) to node [pos=0.6, anchor=north, inner sep=1pt] {$f$} (u2);
\end{tikzpicture}
\]
There is an isomorphism $\Tt C^*(E) \to \Tt C^*(F)$ that carries $q_v$ to $q_v - t_e t^*_e$,
carries $q_u$ to $q_u + t_e t^*_e$ and takes each of the remaining generators of $\Tt C^*(E)$ to
the generator of $\Tt C^*(F)$ with the same label. This isomorphism intertwines $\kappa^E$ and
$\kappa^F$ because in both graphs every edge has source $w$. It does not, however, carry $M_E$ to
$M_F$ since, for example, $q_v - t_e t^*_e \not\in M_F$.
\end{exm}

\end{document}